\documentclass[a4paper]{article}
\usepackage[english]{babel}
\usepackage[utf8x]{inputenc}
\usepackage[T1]{fontenc}
\usepackage{comment}

\newcommand{\ignore}[1]{}

\usepackage[a4paper,top=3cm,bottom=2cm,left=3cm,right=3cm,marginparwidth=1.75cm]{geometry}

\usepackage{amsmath}
\usepackage{amsthm}
\usepackage{amssymb}
\usepackage{graphicx}
\usepackage[colorinlistoftodos]{todonotes}
\usepackage[colorlinks=true, allcolors=blue]{hyperref}
\usepackage{mathtools}
\usepackage{bm}

\newtheorem{theorem}{Theorem}[section]
\newtheorem{proposition}[theorem]{Proposition}

\newtheorem{corollary}[theorem]{Corollary}

\newtheorem{conjecture}[theorem]{Conjecture}
\newtheorem{definition}[theorem]{Definition}

\newtheorem{problem}[theorem]{Problem}
\newtheorem{remark}[theorem]{Remark}

\title{Asymptotically Almost Every $2r$-regular Graph has an Internal Partition}
\author{Nathan Linial\thanks{Department of Computer Science, Hebrew University, Jerusalem 91904,
    Israel. e-mail: nati@cs.huji.ac.il~. Supported by ISF Grant 1169/14, Local and Global Combinatorics.} \and  {Sria Louis\thanks{Department of Computer Science, Hebrew University, Jerusalem 91904,
    Israel. e-mail: sria.louis@cs.huji.ac.il~.}}}

\begin{document}
\maketitle

\begin{abstract}
An {\em internal partition} of a graph is a partitioning of the vertex set into two parts such that for every vertex, at least half of its neighbors are on its side.
We prove that for every positive integer $r$, asymptotically almost every $2r$-regular graph has an internal partition.
\end{abstract}

\section{Introduction} \label{problem_description}
\subsection{Notations}

We denote the neighbor set of a vertex $x$ in a graph $G=(V,E)$ by $N(x)$ and its degree by $d(x)$. The number of neighbors that $x$ has in a subset $A \subseteq V$ is denoted by $d_{A}(x) = |N(x)\cap A|$. We denote by $\mathbb{N}_{\ge k}$ the set $\{k, k+1,\ldots\}$.
For a vertex $x$ and a set $A$ we use the shorthand notation $A\cup x$ and $A\backslash x$ rather than $A\cup \{x\}$ and $A\backslash \{x\}$ respectively.

Our main concern is with \textit{partitions} of $V$ into two parts $\left\langle A,B\right\rangle$. We denote by $e(A, B)$ the {\em cut size} of this partition, i.e., the number of edges $e=(x,y)$ with $x\in A$ and $y \in B$.

\subsection{Internal Partitions}

Let $G=(V,E)$ be a simple undirected graph. A partition $\left\langle A,B\right\rangle$ of $V$ is called {\em external} if every vertex has at least as many neighbors on the other side as it has on its own side. Clearly, every graph has an external partition, e.g., one that maximizes the cut size $e(A,B)$. Likewise, in an {\em internal partition} every vertex has at least as many neighbors on its own side as on the other side. This requirement is clearly satisfied by the trivial partition $\left\langle \emptyset ,V\right\rangle$, but we insist on a \textit{non-trivial} internal partition where both parts are non-empty.

Yet, it is worthwhile to consider the following more general definition:

\begin{definition}
Let $G=(V,E)$ be a graph and let $a,b:V\rightarrow \mathbb{N}$ be two functions. We say that a partition $\left\langle A,B\right\rangle$ of $V$ is {\em $(a,b)$-internal}, if:
\begin{enumerate}
\item $d_A(x)\geq a(x)$ for every $x \in A$, and
\item $d_B(x)\geq b(x)$ for every $x \in B$.
\end{enumerate}
\end{definition}

In these terms, an \textit{internal partition} as mentioned above is the same as a $\left(\lceil{\frac{d(x)}{2}}\rceil,\lceil{\frac{d(x)}{2}}\rceil\right)$-internal partition.

The problem of the existence of, or the efficient finding of, internal partitions appears in the literature under various names: \textit{Decomposition under Degree Constraints} \cite{Thomassen1983},  {\em Cohesive Subsets} \cite{Moris2000}, \textit{$q$-internal Partition} \cite{BanLinial} and \textit{Satisfactory Graph Partition} \cite{Gerber2000}. A related subject is \textit{Defensive Alliance Partition Number} of a graph.  A short survey about \textit{alliance} and \textit{friendly} partitions, and the link to the internal partitions, can be found in \cite{ErohGera}. A generalization (and a survey) of \textit{Alliance Partitions} can be found in \cite{FernauRodriguez}.

It is also interesting to reformulate, saying that $\left\langle A,B\right\rangle$ is an $(a,b)$-internal partition iff \\ $d_B(x) \le d(x)-a(x)$ for every $x\in A$ and $d_A(y) \le d(y)-b(y)$ for every $y\in B$. From this perspective, the existence of an $(a,b)$-internal partition can be viewed as a strong {\em isoperimetric inequality} for $G$. Recall that the isoperimetric number (or Cheeger constant) of a graph is $\min\frac{e(A,B)}{|A|}$ over all partitions $\left\langle A,B\right\rangle$ and where $|A|\le |B|$. That is, it is defined by a set $A\subset V$ which minimizes the {\em average} $\frac{1}{|A|}\sum_{x\in A} d_B(x)$, whereas in considering $(a,b)$-internal partitions, we bound $d_B(x)$ for {\em each} vertex in $A$. We return to this perspective in the open problems section.

While every graph has an external partition, there are simple examples of graphs with no \textit{internal} partitions, e.g., cliques or complete bipartite graphs in which at least one part has odd cardinality. On the other hand, it is not easy to find large sparse graphs that have no internal partition. Also, as some of the theorems next mentioned show, {\em nearly} internal partitions (the exact meaning of this is clarified below) always exist.
\\
Stiebitz \cite{Stiebitz1996}, responding to a problem of Thomassen \cite{Thomassen1983}, made a breakthrough in this area. His results and later work by others in a similar vein are summarized in the following theorem.
\begin{theorem}\label{StiebitzKanekoDiwan}
Let $G=(V,E)$ be a graph and let $a,b:V(G)\rightarrow \mathbb{N}$. Each of the following conditions implies the existence of an $(a,b)$-internal partition:
\begin{enumerate}
\item \cite{Stiebitz1996} $d(x)\geq a(x)+b(x)+1$ for every $x\in V$.
\item \cite{Kaneko1998} $d(x)\geq a(x)+b(x)$ for every $x\in V$ and $G$ is triangle-free.
\item \cite{Diwan}\cite{Gerber2004} $d(x)\geq a(x)+b(x)-1$  for  every $x\in V$ and $\text{girth}(G)\ge 5$.
\end{enumerate}
\end{theorem}
A very recent manuscript by Ma and Yang \cite{MaYang} shows that in the last statement of the theorem it suffices to assume that $G$ is $C_4$-free.

Further work in this area falls into several main categories:

\begin{itemize}
\item {\bf How hard is it to decide whether a given graph has an $(a,b)$-internal partition?} This question is investigated in a series of papers by Bazgan et al., surveyed in \cite{Bazgan2010}. Each existence statement in theorem \ref{StiebitzKanekoDiwan} comes with a polynomial time algorithm to find the promised partition. For large $a$ and $b$ the problem seems to become difficult. For example, a theorem of Chvátal \cite{Chvatal} says that the case $a=b=d-1$ is NP-hard for graphs in which all vertex degrees are $3$ and $4$.
\item {\bf Generalizations and variations:} Gerber and Kobler \cite{Gerber2000} introduced vertex- and edge-weighted versions of the problem and showed that these are NP-complete. Recent works by Ban \cite{Ban2016} and by Schweser and Stiebitz \cite{StiebitzSchweser2017} extend theorem \ref{StiebitzKanekoDiwan} to edge-weighted graphs.
    It is NP-hard to decide the existence of an internal {\em bisection} i.e., an internal partition with $ |A|=|B|$ \cite{Bazgan2006}. There is literature concerning approximate internal partitions, partitions into more than two parts etc. See \cite{Gerber2003}, \cite{Gerber2004} , \cite{Bazgan2005}, \cite{Bazgan2010}.
\item {\bf Sufficient conditions:} Shafique and Dutton \cite{Shafique2002} and \cite{Gerber2004} provide several sufficient conditions for the existence of an internal partition in general graphs, and more specific conditions for line-graphs and triangle-free graphs.
\item {\bf Necessary conditions:} It is proved in \cite{Shafique2002} that there is no forbidden subgraph characterization for the existence or non-existence of an internal partition. Given a graph's edge-density it is possible to bound the cardinality of the parts of an internal partition, if it exists \cite{Gerber2000}.
\item {\bf Regular graphs:} For $d=3,4$ the only $d$-regular graphs with no internal partition are $K_{3,3}, K_4, K_5$  \cite{Shafique2002}. As shown by Ban and Linial \cite{BanLinial}, every $6$-regular graph with $14$ or more vertices has an internal partition. The case of $5$-regular graphs remains open.

\end{itemize}
The most comprehensive survey of the subject of which we know is \cite{Bazgan2010}.

\section{The Theorem : 2r-regular Graphs} \label{thm}

As mentioned, the repertoire of graphs with no internal partitions seems rather limited, and for $d\in \{ 3,4,6\}$ there are only finitely many $d$-regular graphs with no internal partition. This has led to the following conjecture \cite{BanLinial}:

\begin{conjecture}\label{finite_number}
For every $d$ there are only finitely many $d$-regular graphs with no internal partitions.
\end{conjecture}

The main theorem of this paper is a weaker version of this conjecture, namely, an asymptotic result for an even $d$.

We say that a graph $G$ is {\em $4$-sparse} if every set of four vertices spans at most four edges (i.e., has no $K_4$ and no diamond-graph as subgraph). We prove:

\begin{theorem}\label{main:thm}
Let $G$ be a graph and let $a,b:V(G)\rightarrow \mathbb{N}_{\geq 2}$ be such that $d(x)\geq a(x)+b(x)$ for every vertex $x\in V$. If $G$ is $4$-sparse then it has an $(a,b)$-internal partition.
\end{theorem}

\begin{corollary}
If $G$ is a $4$-sparse graph and all its vertices have even degrees, then $G$ has an internal partition.
\end{corollary}

We note the following simple fact about random regular graphs:

\begin{proposition}
For every $d\ge 3$ asymptotically almost every $d$-regular graph is $4$-sparse.
\end{proposition}
\begin{proof}
We work with the configuration model of random $n$-vertex $d$-regular graphs. Let $X$ be the random variable that counts the number of sets of four vertices that span five or six edges. Then $\mathbb{E}(X)\le O(n^4)\cdot\frac{(dn-11)!!}{(dn-1)!!}=O(\frac 1n)$. 
\end{proof}


We can now conclude the theorem in the title of this paper, namely:

\begin{corollary}
Asymptotically almost every $2r$-regular graph has an internal partition.
\end{corollary}

Before we prove theorem \ref{main:thm}, we need to introduce several definitions. 

\begin{definition}
Let $G=(V,E)$ be a graph, and let $f:V\rightarrow \mathbb N$.
\begin{itemize}
\item We say that $A\subseteq V$ is \textit{$f$-internal} if $d_A(x)\geq f(x)$ for every $x \in A$.
\item We say that $A\subseteq V$ is \textit{$f$-degenerate} if every non-empty subset $K \subseteq A$ has a vertex $x\in K$  such that $d_K(x)\leq f(x)$.
\end{itemize}
\end{definition}

\begin{remark}\label{internal_subset}
Clearly, a non-empty set is not $(a-1)$-degenerate if and only if it contains a non-empty $a$-internal subset.
\end{remark}

\begin{definition}

Let $A,B\subseteq V$ be non-empty disjoint sets, and let $f,g:A\rightarrow \mathbb N$. The pair $\left\langle A,B\right\rangle$ is said to be
\begin{itemize}
\item \textit{$(f,g)$-internal} if $A$ is $f$-internal and $B$ is $g$-internal.
\item \textit{$(f,g)$-degenerate} if $A$ is $f$-degenerate and $B$ is $g$-degenerate.
\end{itemize}
It is an \textit{$(f,g)$-degenerate partition} if it is an $(f,g)$-degenerate pair and, in addition, $B=V\backslash A$.

\end{definition}

We claim next that for every $(a,b)$-degenerate partition $\left\langle A,B\right\rangle$ it holds that $|A|\geq 2$ (and, by symmetry, also $|B|\geq 2$). Otherwise, every vertex $x\in B$ satisfies $d_B(x)=d(x)-d_A(x)\geq d(x)-1> d(x)-a(x) = b(x)$, so that $B$ is not $b$-degenerate. The last inequality uses the assumption that $a(v)\geq 2$ for all $v\in V$.

\subsection*{Proof of Theorem \ref{main:thm}}
For the proof of \ref{main:thm} it clearly suffices to consider the case where $d(x) = a(x)+b(x)$ for every $x\in V$. Our proof is based on the methodology initiated by Stiebitz \cite{Stiebitz1996}. The proof is by contradiction. We assume that $G$ is a counterexample, i.e., a $4$-sparse $d$-regular graph with no $(a,b)$-internal partition.  
\\

We next define an objective function on vertex partitions of the graph. By optimizing it, we will achieve a contradiction.\\

For a function $f:V\to \mathbb N$ and $S\subseteq V$, we denote $f(S):=\sum_{x\in S} f(x)$. We make substantial use of the following function $w$ that is
defined for every partition $\left\langle A,B\right\rangle$ of $V$ as follows:
\[w(A,B)=a(B)+b(A)-e(A,B)\]
We calculate the change in $w$ when one vertex changes sides. If $\left\langle A',B'\right\rangle = \left\langle A \cup x,B\backslash x \right\rangle$, then
\[w(A',B') - w(A,B) = 2 \left(b(x)-d_B(x) \right) \]
and if $\left\langle A'',B''\right\rangle =  \left\langle A \backslash y,B\cup y \right\rangle$, then
\[w(A'',B'') - w(A,B) = 2 \left( a(y)-d_A(y) \right)\] \\
We denote by $\Delta w$ the change in $w$ and call a partition of $V$ {\em locally maximal} if $\Delta w\le 0$ whenever a single vertex moves from one part to the other.
\\
For the sake of completeness, we reproduce two easy but crucial propositions, following Stiebitz.

\begin{proposition}\label{AB_disjoint_internal}
If $G$ has an $(a,b)$-internal pair, then it has an $(a,b)$-internal partition.
\end{proposition}
\begin{proof}
Consider an $(a,b)$-internal pair of sets $A,B\subset V$, that maximizes $|A|+|B|$. If $\left\langle A,B\right\rangle$ is not an internal partition, then $U:=V\backslash(A\cup B)$ is non-empty. By maximality, for every vertex $x\in U$ it holds that $d_A(x)\le a(x)-1$, and hence $d_{B\cup U}(x)\ge b(x)+1$. Thus the partition $\left\langle A,B\cup U\right\rangle$ is $(a,b)$-internal contrary to the assumed maximality.
\end{proof}

\begin{proposition}\label{collapsibility}
If $G$ has no $(a,b)$-internal partition, and if $ A\subsetneq V$ is not $(a-1)$-degenerate, then $V\backslash A$ is $(b-1)$-degenerate.
\end{proposition}
\begin{proof}

Since $A$ is not $(a-1)$-degenerate it has a non-empty $a$-internal subset. If $V\backslash A$ is not $(b-1)$-degenerate, then it contains a non-empty $b$-internal subset. This gives an $(a,b)$-internal pair, contrary to Proposition \ref{AB_disjoint_internal}.
\end{proof}

Consider the family of non-empty sets $A\subsetneq V$ that are $a$-degenerate, but not $(a-1)$-degenerate. This family is not empty (for instance, take $A$ to be an inclusion-minimal $a$-internal subset). Among this family of sets, take a set $A$ such that
\begin{itemize}
\item the partition $\left\langle A,V\backslash A\right\rangle$ is locally maximal for $w$, and
\item minimizes $|A|$ under these assumptions.
\end{itemize}

We claim now that the resulting set $A$ {\em is} $a$-internal. Otherwise, there is $v\in A$ such that $d_A(v)\leq a(v)-1$. Then, however, $A \backslash {v}$ is non-empty $a$-degenerate and not $(a-1)$-degenerate. In addition, upon moving $v$ from $A$ to $B=V\backslash A$ it holds that $\Delta w\geq 2$, contradicting the maximality of $w$.

We next consider the vertices of "low internal-degree" in $A$ and in $B$.
Denote \[C = \{v\in A ~|~ d_A(v) = a(v) \}~~~ \text{and}  ~~~D = \{v\in B ~|~ d_B(v) \leq  b(v) -1 \}.\]
Note that $C\neq \emptyset$ since $A$ is $a$-degenerate. Also $D\neq \emptyset$, since $B$ is $(b-1)$-degenerate by Proposition \ref{collapsibility}.
For every $x\in D$ by moving $x$ to $A$ we get $\Delta w \geq 2$. Also, for $y\in C$, by moving $y$ to $B$ we get $\Delta w \geq 0$.

\begin{proposition}
For every $x \in D$ there is a subset $A_x \subseteq A$ such that $A_x \cup x$ is $(a+1)$-internal.
\end{proposition}
\begin{proof}
As mentioned, moving $x\in D$ to $A$ yields $\Delta w \geq 2$. Also, $A\cup x$ is clearly not $(a-1)$-degenerate. Consequently, by the maximality of $w(A,V\backslash A)$,  $A\cup x$ cannot be $a$-degenerate, and it must contain an $(a+1)$-internal subset.  This $(a+1)$-internal subset must contain $x$, as claimed. Note also that $A_x$ is necessarily $a$-internal.
\end{proof}

\begin{proposition}
For every $x\in D$ and for every such set $A_x$ it holds that $C\subseteq A_x$.
\end{proposition}
\begin{proof}
Suppose there is $y \in C\backslash A_x$. Clearly, $A\backslash y$ is $a$-degenerate. Also, $A\backslash y$ is not $(a-1)$-degenerate since it contains the $a$-internal subset $A_x$. In addition $w(A \backslash {y},B\cup {y})=w(A,B)$,
contradicting the minimality of $|A|$.
\end{proof}

\begin{proposition}\label{CxAdjacent}
Every vertex in $C$ is adjacent to every vertex in $D$.
\end{proposition}
\begin{proof}
Consider some $x\in D$ and $y\in C$. Then $d_A(y)=a(y)$. But $y$ also belongs to the $(a+1)$-internal set $A_x\cup x$, so that $d_{A\cup x}(y)=a(y)+1$. The conclusion follows.
\end{proof}

\begin{proposition}\label{yzAdjacent}
There are two adjacent vertices in $C$.
\end{proposition}
\begin{proof}
We know already that $C\neq \emptyset$. Consider some $y\in C$. Clearly $A\backslash y$ is $a$-degenerate. Also $w(A \backslash {y},B\cup {y})=w(A,B)$, therefore, by the minimality of $|A|$, the set $A\backslash y$ must be $(a-1)$-degenerate. In particular there is $z \in A$ such that $d_{A\backslash y}(z) \leq a(x)-1$, whereas $d_A(z)\geq a(x)$. It follows that $z\in C$ and $yz\in E(G)$.
\end{proof}

The above propositions, and the $4$-sparsity of $G$ imply that
\begin{itemize}
\item Neither y nor z have an additional neighbor in C.
\item The only common neighbor of $y$ and $z$ is $x\in D$. 
\end{itemize}

In other words, $d_{A\backslash \{y,z\}}(v)\geq a(v)$ for every $v\in A\backslash \{y,z\}$. This means that $A':=A\backslash \{y,z\}$ is $a$-internal. Therefore, $A'$ is $a$-degenerate, and not $(a-1)$-degenerate. Also, $A'\neq \emptyset$ since $d_A(y)=a(y)\geq 2$. Finally, $w(A',V\backslash A')\geq w(A,V\backslash A)$, contrary to the minimality of $|A|$.\hfill$\qed $

\section{Some Computational and Experimental Results} \label{alg}

\subsection{Algorithmic Realization}

As mentioned, to the best of our knowledge, all previous existence proofs of $(a,b)$-internal partitions translate into polynomial-time algorithms. This applies as well to Theorem \ref{main:thm} and its proof. Thus, let $G$ be a $4$-sparse graph, and assume $d(x)=a(x)+b(x)$  for every $x\in V(G)$. We maintain from the proof of the theorem the definitions of the sets $C$ and $D$. Here is a polynomial-time algorithm that finds an $(a,b)$-internal partition of $G$:

\medskip

{\bf Initialization.} Find a set $A$ that is $a$-degenerate but not $(a-1)$-degenerate. Repeatedly remove any vertex $x\in A$ with $d_A(x)<a(x)$ until none remain.

{\bf Loop.} While $V\backslash A$ is $(b-1)$-degenerate:
\begin{enumerate}
\item If there is $x\in D$ such that $A\cup x$ is $a$-degenerate, update $A:=A\cup x$.
\item Else, there is a triangle $xyz$ with $x\in D$ and $y,z\in C$. Update $A:=A\backslash \{y,z\}$.
\end{enumerate}

Note that the {\sl initialization} can be done in polynomial time, since (i) for any $\varphi: V\to\mathbb{N}$ it takes linear time to check whether a given $S\subseteq V$ is $\varphi$-degenerate, (ii) $G$ is $(a+b)$-internal, (iii) if $S$ is not $(\varphi+1)$-degenerate, then $S\backslash x$ is not $\varphi$-degenerate for every $x\in S$ and (iv) if, moreover, $d(x)\le\varphi(x)$, then $S\backslash x$ is not $(\varphi+1)$-degenerate.\\
The if-else {\sl dichotomy} is proven in Propositions \ref{CxAdjacent} and \ref{yzAdjacent}.\\
The algorithm {\sl terminates}, since in step (1) $\Delta w >0$ , and step (2) decreases $|A|$ while keeping $\Delta w\geq 0$. Termination is proved by induction on lexicographically-ordered pairs $(w,-|A|)$.\\
$A$ remains not-$(a-1)$-degenerate throughout, and upon termination $V\backslash A$ is not $(b-1)$-degenerate. Therefore, $A$ and $V\backslash A$ contain $a$-internal and $b$-internal subsets, respectively, which is, according to proposition (\ref{AB_disjoint_internal}), a sufficient condition for the existence of an internal partition that can be found in polynomial time (See Proposition $1$ in \cite{Bazgan2003TR}).

\subsection{Improving Previous Experimental Results}
In \cite{Gerber2000} some experimental results are presented. They apply a heuristic algorithm in an attempt to find a $\left(\lceil{\frac{d(x)}{2}}\rceil,\lceil{\frac{d(x)}{2}}\rceil\right)$-internal partition in random graphs. Their algorithm starts from a random partition and at each iteration minimizes $f(A,B)=\sum_{v\in A}  (d_A(v)-d_B(v))^++\sum_{v\in B} (d_B(v)-d_A(v))^+$  where the minimum is taken over all partitions which were achieved by switching a single vertex most of whose neighbors are on the other side of the partition. 
The process can terminate with either an internal partition or a trivial partition. It can also loop indefinitely. In the latter two cases, they restart the process.

We have experimented with a similar algorithm. The main change is that we consider only near-bisections $\left\langle A,V\backslash A\right\rangle$, and insist that $\left| |A|- \frac{|V|}{2}\right|\leq c(n)$ for $c(n)=\log_d(n)$. When this condition is violated, we move a random vertex from the big part to the small. This algorithm may either output an internal partition or loop forever. However, in extensive simulations with random $d$-regular graphs ($30 \leq n \leq 10000$ and $4\leq d \leq \min(50,\frac{n}{2})$) the algorithm has always found an internal partition in fewer than $5n$ iterations. This conclusion seems to be hardly affected by the exact choice of $c$.

\section{Discussion and Some Open Problems} \label{summary}

We have mentioned above the analogy between the existence of internal partitions and upper bounds on Cheeger constants. As shown by Alon \cite{Alon1997}, for every large $d$-regular graph the Cheeger constant is at most $\frac d2 -c\sqrt{d}$ for some absolute $c>0$.
Although conjecture \ref{finite_number} is still open, this prompts an even more far-reaching possibility.

\begin{problem}
Is it true that for every integer $\delta\ge 1$ there are integers $d$ and $n_0$ such that every $d$-regular graph on $n>n_0$ vertices has a $(\frac d2 +\delta,\frac d2 +\delta)$-internal partition ?
\end{problem}

Also, the upper bound on Cheeger's constant in Alon's paper is actually attained by a bisection (the two parts differ in size by at most one). This suggests:

\begin{problem}
Does conjecture \ref{finite_number} hold also with "near" bisections? E.g., where the cardinalities of the two parts differ by $O_d(1)$.
\end{problem}

How does the computational complexity of the internal partition vary as $n$ grows? So far, existence theorems have gone hand-in-hand with efficient search algorithms. Is this a coincidence or is there a real phenomenon?

\begin{problem}
Consider the computational complexity of both the decision and the search version of the internal partition problem for $d$-regular graphs on $n$ vertices. Conjecture \ref{finite_number} says that an internal partition exists whenever $n>n_0(d)$. Is there some $n_1(d)$ such that for $n>n_1(d)$ an internal partition can be found efficiently? If so, do $n_0$ and $n_1$ coincide?
\end{problem}

\section*{Acknowledgements}
We thank Amir Ban, David Eisenberg, Zur Luria, Jonathan Mosheiff, and Yuval Peled for careful reading and useful comments.

\bibliographystyle{alpha}

\end{document}